\documentclass[11pt]{article} 
\usepackage{graphicx} 
\usepackage{psfrag} 
\begin{document}

\newcommand{\qed}{\hfill{\setlength{\fboxsep}{0pt}
                  \framebox[7pt]{\rule{0pt}{7pt}}} \newline}
\newcommand{\eqed}{\qquad{\setlength{\fboxsep}{0pt}
                  \framebox[7pt]{\rule{0pt}{7pt}}} }
\newcommand{\st}{\,\colon\,}

\newtheorem{theorem}{Theorem}
\newtheorem{lemma}[theorem]{Lemma}         
\newtheorem{corollary}[theorem]{Corollary}
\newtheorem{proposition}[theorem]{Proposition}
\newtheorem{definition}[theorem]{Definition}
\newtheorem{claim}[theorem]{Claim}
\newtheorem{conjecture}[theorem]{Conjecture}

\newcommand{\proof }{{\bf Proof.\ }}          


\def\SS{S}
\def\rdeg{\mathop{\rm rdeg}\nolimits}


\title{Dominating Sets in Plane Triangulations}  
\author {Erika L.C. King\thanks{Department of Mathematics and Computer Science, Hobart and William Smith Colleges, Geneva, NY 14456, eking@hws.edu}\,,
Michael J. Pelsmajer\thanks{Department of Applied Mathematics,
Illinois Institute of Technology, Chicago, IL 60616, pelsmajer@iit.edu. Partially supported by NSA Grant H98230-08-1-0043.}
}

\date{\today}
\maketitle

\begin{abstract}
In 1996, Matheson and Tarjan conjectured that any $n$-vertex
plane triangulation with $n$ sufficiently large has
a dominating set of size at most $n/4$.  We prove this for graphs of
maximum degree 6.
\end{abstract}

\section{Introduction}
A {\it dominating set} $D\subseteq V$ of a graph $G$ is a set such that each vertex $v\in V$ is either in the set or adjacent to a vertex in
the set. The {\it domination number} of $G$, denoted $\gamma(G)$, is defined as the minimum cardinality of a dominating set of $G$.
In 1996, Matheson and Tarjan \cite{mattar} considered this asymptotically for {\em plane triangulations} (or {\em triangulations}, for short),
motivated by studying unstructured multigrid computations in two dimensions.
They proved that any triangulated disc $G$ with $n$ vertices has $\gamma(G)\leq n/3$. They
defined an infinite class of outerplane graphs which achieve this bound, proving the bound is sharp
(even for outerplane graphs).  However, they conjectured that one can do better if the disc is
bounded by a triangle, that is, for a triangulation.

\begin{conjecture}{\rm \cite{mattar}}\label{conj}
For $n$ sufficiently large, the domination number of any $n$-vertex triangulation is at most $n/4$.
\end{conjecture}

As they note, the octahedron, which is a triangulation, contains six vertices and cannot be dominated
with less than two vertices; hence $n>6$ is necessary.
They also defined an infinite class of triangulations that need $n/4$ vertices to be
dominated, and thus we know we cannot do better. These graphs are constructed from any number of $K_4$s drawn in
the plane, with edges added to the outer face to create a triangulation.

Domination is very widely studied (see~\cite{mono} for a recent monograph),
and upper bounds have been found for various graph classes related to triangulations.
For example, every triangulation other than $K_3$ has minimum degree $3$, $4$, or $5$, and graph classes
achieving these three minimum degree values have been studied.
Reed~\cite{Reed96} proved that every $n$-vertex graph of minimum degree $3$
has a dominating set of size at most $3n/8$.
Xing, Sun, and Chen~\cite{five} proved that minimum degree at least $5$ implies
that there is a dominating set of size at most $5n/14$, and
Sohn and Yuan~\cite{SohnYuan} show that with
minimum degree $4$ there must be a dominating set of size at most $4n/11$.
However,
none of these bounds for general graphs is as low as Matheson and Tarjan's bound of $n/3$ for triangulated discs.
As for planar graphs other than triangulations, people have primarily investigated
upper bounds for planar graphs of small diameter~\cite{GodHenII,GodHenI,MacG}
and recently there has also been a lot of work 
from the computational point of view (for example,~\cite{Fomin}). 
Recently Honjo, Kawarabayashi, and Nakamoto considered triangulations on other surfaces, extending Matheson and Tarjan's bound of $n/3$
to triangulations on the projective plane, the torus, and the Klein bottle, and to locally planar triangulations
(triangulations of sufficiently high representativity) for every other surface~\cite{HKN}.

When trying to construct a small dominating set for a given graph, it is rather intuitive that
vertices of relatively high degree are good candidates for being part of a dominating set.
The hard part ought to be in efficiently dominating those vertices that are not
adjacent to any high degree vertices.
However, this approach is not helpful for graphs with small maximum degree, which includes
many triangulations, for example, the triangulations in which all vertex degrees are $5$ or $6$.
(These are known as {\em geodesic domes}, or as the duals of {\em fullerene graphs}, and they
are of interest~\cite{Graver} partly due to applications in chemistry.)
Moreover, such graphs must be considered at some point when attacking Conjecture~\ref{conj}.
Therefore it is interesting to see what kind of approach will work on such graphs, and thus this case is our focus
for this paper. Our main result is the following theorem.

\begin{theorem}\label{T:main}
There exists $n_0$ such that for any $n\ge n_0$, an $n$-vertex triangulation with maximum degree $6$
has a dominating set of size at most $n/4$.
\end{theorem}

Note that
since a triangulation with maximum degree less than $6$ has at most $12$ vertices,
Theorem~\ref{T:main} could be restated to include triangulations with no degree greater than 6.

Section~\ref{S:extend} considers ways we can extend Theorem~\ref{T:main}.
Sections~\ref{S:defs} and~\ref{S:proof}
are devoted to the proof of Theorem~\ref{T:main}.
As there are a fair amount of details to work
through and verify, we first give an outline of its main steps.

\section{Sketch of the proof}
In this sketch, statements and quantities are only approximate, and claims are mostly unjustified.

Let $G$ be an $n$-vertex triangulation of maximum degree 6.
Let $U$ be the set of vertices in $G$ of degree less than 6; using
Euler's formula we can see that $|U|\le 12$.  We select a
minimum-size tree $T$ in $G$ that contains $U$ (a ``Steiner tree'').
Then, thinking of $G$
as being embedded on the sphere rather than the plane, we will cut
the surface along the edges of $T$, obtaining a triangulated disc.
Let $G'$ be the plane graph obtained; note that the edges of $G'$ along
the boundary of the disc consist of two copies of each edge of $T$.

Now let $G_\infty$ be the 6-regular infinite triangulation, as
suggested on the left of Figure~\ref{F:infinite-grid}.
The triangulated disc can be embedded in $G_\infty$ such
that vertices, edges, and 3-faces are mapped to their counterparts in $G_\infty$,
preserving incidences.  $G_\infty$ has a dominating set
arranged in a pattern that uses every seventh vertex.
We copy this pattern of vertices to $G'$, then to $G$, then add all the
vertices of $T$ to get a dominating set for $G$.  Its size is roughly $|V(G')|/7+|V(T)|$, which is
about $(n+8|V(T)|)/7$.  This suffices if $|V(T)|$ is at most $3n/32$.

\begin{figure}[ht]
\begin{center}
\includegraphics[width=.8\textwidth]{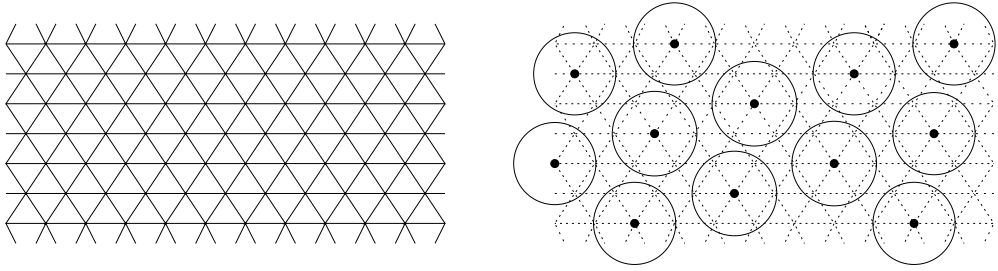}
\end{center}
\caption{$G_\infty$, with a dominating set that contains every seventh vertex}
\label{F:infinite-grid}
\end{figure}

Next, we show that when $T$ contains more than $3n/32$ vertices of $G$, it must be that
the triangulation $G$ of the sphere mostly consists of a
``triangulated cylinder'' with length $\ell$ much larger than
its width $w$, as follows:

Let $U'=U\cup \{v\in V(T)\st d_T(v)\not=2\}$.  By the choice of $T$, $|U'|\le 2|U|$.
There are $|U'|-1$ paths in $T$ with endpoints in $U'$ and internal vertices in $V(T)-U'$;
let $P$ be a longest one. Then the length of $P$ is at least $|E(T)|/(|U'|-1)$, which is
at least $|V(T)|/24$.  Let $x$ be a middle vertex of $P$.  By the choice of $T$ and $P$,
the vertices of $U'$ closest to $x$ are the endpoints of $P$.  Therefore, the distance from $x$
to $U$ is at least $|V(T)|/48$, which is at least $(3n/32)/48=n/512$.

For each
$i\ge 0$, let $G_i$ be the subgraph induced by vertices of distance at most $i$
from $x$.  Let $r$ be smallest such that $G_r$ is not a triangulated hexagon.
(See Figure~\ref{F:hex-grid}.)
$G_{r-1}$ contains about $3r^2$ distinct vertices, so $r\le \sqrt{n/3}$.
Since this is less than $n/512$ (for $n$ sufficiently large),
$G_r$ contains no vertex of $U$; all its vertices have
degree~6 in $G$.  Then, by the choice of $r$, there must be two vertices $y,y'$ on different
sides of the hexagon $G_{r-1}$ that are connected by a short path (length 2 or~3)
within $G_r$.  This yields a cycle $C$ in $G_r$ that contains
$x$, $y$, and $y'$, of length at most $2r+1$.
Since the distance from $x$ to $U$ is at least $n/512$,
and the distance from $x$ to any vertex of $C$ is at most $\sqrt{n/3}$,
the distance from $C$ to $U$ is at least $n/512 - \sqrt{n/3}$, which is at least $n/513$
(for $n$ sufficiently large).

Any cycle embedded in the plane or sphere has two sides.
Let $C_i$ be the graph induced by vertices on one side of $C$ that are at distance $i$ from $C$,
and let $C_{-i}$ be the graph induced by vertices on the other side of $C$ that are at distance $i$ from $C$, for all $i\ge 0$.
Let $j,j'$ be smallest such that $U$ intersects $C_j$ and $C_{-j'}$.
Then $\min(j,j')\ge n/513$ and every vertex in $C_i$ has degree~6 in $G$ for $-j'< i<j$.
Each $C_i$ with $-j'< i< j$ is a cycle.
Also, $|C_i|-|C_{i+1}|=k$ for all $-j'< i < j-1$, for some
constant $k\in\{-2,-1,0,1,2\}$ that depends only on the location of $y$ and $y'$ on $G_{r-1}$.
If $k\not=0$, then for $0\le i\le \min(j,j')$ we have $\min(|C_i|,|C_{-i}|)= |C|-i$.
Then $C_{2r+1}$ is empty since $|C|\le 2r+1$,
but $C_{2r+1}$ is a cycle since $2r+1<<n/513\le j$, a contradiction.
Therefore $k=0$ and $|C_i|=|C|$ for $-j'< i< j$.
Such cycles and the edges between consecutive cycles form
a triangulated cylinder of ``length'' $\ell\ge j+j'-2$ and ``width'' $w=|C|$.  Since $j+j'\ge 2n/513$ and $|C|$ is
at most
$2\sqrt{n/3}+1$, the length is much larger than the width.
Removing the triangulated cylinder yields two components.  Using the fact that $U$ intersects $C_j$ and $C_{j'}$, we show that
each component has at most $|C|^2/2$ vertices, which together is at most $w^2$ vertices.

To finish, we do something like cutting the triangulated
cylinder lengthwise, mapping it to $G_\infty$ and copying
the dominating pattern back to the cylinder. This produces a set of
size approximately $\ell(w+2)/7$ that dominates the cylinder.
The rest of $G$ has at most $w^2$ vertices,
which we add to the previous set, obtaining a set that dominates $G$.
Its size is about $\ell(w+2)/7+w^2$, which is roughly $\ell(w+2)/7$
since $\ell>>w$.
Also $n$ is approximately $\ell w +w^2 = w(\ell+w)$,
which is roughly $w \ell$ for the same reason. Thus the result follows
if $\ell(w+2)/7<\ell w/4$, which is true because $w\ge 3$.

This sketch overlooks one major issue.
In order to discuss the structure and number of vertices on each side of a cycle (and more),
we need to introduce the concept of {\em marginal degree} and prove several results about it.
This takes place in Section~\ref{S:defs}.

\section{Definitions and preliminaries}\label{S:defs}

We make our definitions for finite graphs.  (Although we will consider
one infinite graph ($G_\infty$), this will cause no confusion.)
For any graph $G$, let $n(G),e(G)$ be the number of vertices and edges, respectively.
For any $S\subseteq V(G)$, let $G[S]$ denote the subgraph induced by $S$.
For a walk $W=v_0,\ldots,v_k$ (open or closed, possibly a path or a cycle)
let $|W|$ denote its {\em length}, which equals $k$; this is the number of
edges (counting multiplicity).

A finite graph is {\it planar} if it can be drawn in the plane (with the
usual restrictions, see for example~\cite{west}); a {\it plane graph} has a fixed drawing.
If we remove the plane graph from the plane then each maximal
connected region is an open set; these are the {\em faces}.  A plane graph has one unbounded face,
called its {\em outer face}; other faces are {\em internal faces}.
An {\it outerplane graph} is a plane graph such that every vertex is incident to the outer face.

Each face of a plane graph is bounded by a set of disjoint walks, its {\em boundary}.
Note that for an outerplane graph, the boundaries of internal faces are precisely the
induced cycles of the graph.
A {\em $k$-face} is a face bounded by one closed walk of length $k$.
A {\em triangulation} is a plane graph in which every face is a $3$-face.
A {\em triangle} is a subgraph isomorphic to $K_3$.

Next we give some original definitions,
intended for the case that most vertices have degree~6.
We also provide lemmas that show how these definitions are used.

Consider a walk $W=v_0,\ldots,v_k$ in a plane graph $G$ (so $k=|W|$).
For each $i$ with $0<i<k$,
count the number of edges incident to $v_i$ from the right---more specifically,
if we order the edges incident to $v_i$ such that they are counterclockwise near $v_i$, then
count the ones after $v_iv_{i-1}$ and before $v_iv_{i+1}$---and denote this quantity $\rdeg_i(W)$.
Let $\overline{W}$ be $W$ in reverse order, so its $j$th vertex is $v_{k-j}$.  See Figure~\ref{F:example1}.
Define the {\it marginal degree of $v_i$ in $W$} to be $\rdeg_i^*(W) = \rdeg_i(W)-2$.
(For some motivation, consider $0<i<k$ such that $v_i$ has degree 6
with its incident edges drawn symmetrically around $v_i$;
then
(1) $\rdeg_i^*(W)=0$\ if and only if\
$\rdeg_{k-i}^*(\overline{W})=0$\  if and only if\  the edges $v_{i-1}v_i$ and $v_iv_{i+1}$
are colinear, and (2) $\rdeg_i^*(W)= -\rdeg_{k-i}^*(\overline{W})$ unless $\rdeg_i^*(W)=3$.)

\begin{figure}[ht]
\begin{center}
\psfrag{W}{$W$}
\psfrag{v1}{$v_{i}$}
\psfrag{v2}{$v_{i+1}$}
\includegraphics{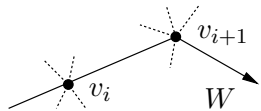}
\end{center}
\caption{An example with $\rdeg_i(W)=\rdeg_{k-i}(\overline{W})=2$,
  $\rdeg_{i+1}(W)=1$, and $\rdeg_{k-i-1}(\overline{W})=3$}
\label{F:example1}
\end{figure}

If $W$ is a closed walk (i.e., $v_k=v_0$) then we think of the indices modulo $k$ and
define $\rdeg_0(W)$ and $\rdeg_0^*(W)$ similarly using the subwalk $v_{k-1},v_0,v_1$.
For a closed walk $W$ we define $\rdeg^*(W)=\sum_{i=0}^{k-1}\rdeg_i^*(W)$,
which we call the
{\it marginal degree of $W$}.
For any connected outerplane subgraph $H$ of a plane graph $G$ with $|V(H)|\ge 2$,
its boundary is a nontrivial closed walk $W$, and without loss of generality, $W$ will
be directed so that at each vertex of $W$, the outer face of $H$ is to the right.
Thus, each vertex along the boundary has a well-defined marginal degree.  Moreover,
since $\rdeg^*(W)$ is unchanged no matter which vertex of $W$ is its start/end,
we may define the {\em marginal degree of $H$} to be $\rdeg^*(H)=\rdeg^*(W)$.

\begin{figure}[ht]
\begin{center}
\includegraphics[width=125pt]{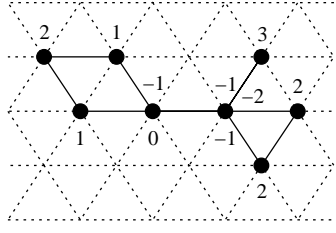}
\end{center}
\caption{A connected outerplane subgraph $H$ of $G_\infty$ with vertices of the boundary walk labeled by their marginal degrees}
\label{F:S-of-C}
\end{figure}

Say that a vertex $v_i$ on $W$ is a {\em turn} if $\rdeg_i^*(W)\not=0$; a {\em left turn} if
positive, and a {\em right turn} if negative.  Note that if $\deg(v_i)=6$, then $v_i$ is a
right turn with respect to $W$ if and only if the $(|W|-i)$th vertex of $\overline{W}$ is a left turn.
Also, $v_i$ is a {\em sharp turn} if the edges $v_iv_{i-1},v_{i}v_{i+1}$ are consecutive around $v_i$.  If $\deg(v_i)=6$,
then this occurs if and only if $|\rdeg_i^*(W)|=2$.  Note that for an outerplane subgraph $H$ of $G_{\infty}$ with boundary walk $W$,
we have $-2 \le \rdeg_i^*(W)\le 3$ for all $0\le i< |W|$.  See Figure \ref{F:S-of-C}.

For an outerplane subgraph $H$ in a plane graph $G$, the {\em interior of
$H$} is the plane minus the boundary and outer face of $H$.

\begin{lemma} \label{L:inf}
Suppose that $H$ is a connected outerplane subgraph of a plane triangulation $G$
with $|V(H)|\ge 2$.
Let $S$ be the union of $V(H)$ and the vertices in the interior of $H$
(i.e., all vertices of $G$ except those that lie in the outer face of $H$).
Then the marginal degree of $H$ is $6-\sum_{v\in S}(6-\deg(v))$.
\end{lemma}

\begin{proof}
The proof proceeds by induction on the number of faces of $G$ in the interior of $H$.
Let $W=v_0,\ldots,v_k(=v_0)$ be the boundary of $H$.

We begin with a separate inductive proof for trees only.
If $|V(H)|=2$ then $|W|=2$ and
each vertex $v_i$ has marginal degree $\deg(v_i)-3$, and
$(\deg(v_0)-3)+(\deg(v_1)-3)=6-\sum_{i\in \{0,1\}}(6-\deg(v_i))$ as desired.
Suppose that $H$ is a tree with $|V(H)|\ge 3$.  We may assume that $v_{k-1}$ is a leaf, in which case $H-v_{k-1}$ is a tree with boundary walk
$W'=v_0,\ldots,v_{k-2}(=v_0)$.  Then
$\rdeg_0^*(W') = (\rdeg_{k-2}(W) + \rdeg_{0}(W)+1) -2$,
$\rdeg_{k-1}^*(W)=\deg(v_{k-1})-3$, and the change in marginal degree, $\rdeg^*(H-v_{k-1})-\rdeg^*(H)$, is
$\rdeg_0^*(W') -\rdeg_{k-2}^*(W) -\rdeg_{k-1}^*(W) -\rdeg_{0}^*(W)$.
Since $\rdeg_i^*(W)=\rdeg_i(W)-2$ for all $i$, the change in marginal degree is $6-\deg(v_{k-1})$.
Then we apply induction to $H-v_{k-1}$, which finishes the proof for trees.

If $H$ is not a tree, then it has nonempty interior $R$ which
contains at least one face of $G$. Pick a $3$-face in $R$ that
shares an edge with $H$, and replace that edge in $W$ by the rest of the
boundary of the $3$-face.  The marginal degree of the vertex added to the
walk is $-2$ and the marginal degree of its
neighbors each increase by one, so the marginal degree of the new walk $W'$ is the same as the
marginal degree of $W$.  Also, $W'$ bounds a connected outerplane subgraph $H'$ with one less $3$-face
in its interior than $H$, so we can apply induction to $H'$ to get its marginal degree, which is equal to the marginal degree of $H$.  Since the outer faces of $H$ and $H'$ contain the exact same set of vertices of $G$, applying induction gives the desired result.
\qed
\end{proof}

When $H$ has only a single vertex $v$, we say that $H$ is bounded
by a walk of length 0 (which is simply $v$), and we define
$\rdeg^*(H)=\deg(v)$.  This makes Lemma~\ref{L:inf} true even in
the case that $|V(H)|\not\ge 2$.

\begin{lemma} \label{C:rdeg}
Suppose that $G$ is a plane triangulation and
$C$ is a cycle in $G$ with no chords on its interior.
Let $H$ be the outerplane graph induced by the neighbors of $C$ which lie in the interior of $C$,
and suppose that $H\not=\emptyset$.

Then $H$ is connected,
and the boundary of $H$ has length $|C|-\rdeg^*(H)$,
and $\rdeg^*(H)=\rdeg^*(C)+\sum_{v\in C}(6-\deg(v))$.
\end{lemma}

\begin{proof}
If $H$ has two components $H_1,H_2$, then there must be a sequence of
$3$-faces interior to $C$ that lead from $H_1$ to $H_2$.  However, since $C$
has no chords, it can be seen that the edges of these $3$-faces contain
a path with no vertex in $C$.  This contradicts $H_1$ and $H_2$ being
disconnected, so $H$ must be connected.

If $H$ has only one vertex $v$, then it is bounded by a
walk of length $0$ and $|C|=\deg(v)$, which implies the length of the walk is $|C|-\rdeg^*(H)$.
Otherwise
let $W$ be the walk that bounds the outer face of $H$,
and we may assume that it is oriented so that its exterior (and $C$)
is always to its right.
(As its own boundary, $C$ is a closed walk oriented so that its exterior is to the right and $W$ is to the left.)
Now $\rdeg_i(W)\ge 1$ for each $0\le i<|W|$, since otherwise the $i$th vertex of $W$ could
not be a neighbor of $C$.
Thus we may double count the 3-faces with a vertex in $W$ and an edge in $C$ to get
$\sum_{i=1}^{|W|}(\rdeg_i(W)-1)=|E(C)|=|V(C)|$.  Thus we have $\rdeg^*(H)=
|V(C)|-|W|$.

For the last part, apply Lemma~\ref{L:inf} to $C$ and $H$.
\qed
\end{proof}

\begin{lemma}\label{L:thank goodness}
Suppose that $G$ is a plane triangulation of maximum degree at most 6 and
$H$ is a connected outerplane subgraph of $G$.  For each $i\ge 0$,
let $V_i$ denote the set of vertices in $H$ or in its interior
that are at distance $i$ from $H$.  Then \[|V_i|\le
\max\{0, |V(H)|-i\cdot\rdeg^*(H) \}.\]
\end{lemma}

\begin{proof}
We proceed by induction on $i$.  The lemma is trivial for $i=0$;
assume that $i\ge 1$.  We may also assume that $V_i\not=\emptyset$.

Note that $V_1$ is the set of vertices
incident to the outer face of $G[\bigcup_{j\ge 1}V_j]$; therefore $G[V_1]$ is an outerplane
graph.  Let $\SS$ be the set of components of $G[V_1]$.

Each vertex $v\in V_i$ is
in $A$ or in the interior of $A$ for some $A\in \SS$,
and the distance from $v$ to $A$ is exactly $i-1$.
Consider an arbitrary component $A\in \SS$.
By induction, the number of vertices
in $A$ or in its interior that are at distance $i-1$ from $A$ is at most
$\max\{0, |V(A)|-(i-1)\rdeg^*(A) \}$.
Therefore
$|V_i|\le \sum_{A\in \SS} \max\{0, |V(A)|-(i-1)\rdeg^*(A) \}$.

Since $A$ lies in a single face of $G[V_0]$, we may
let $C(A)$ be the cycle that bounds that face.
By Lemma~\ref{C:rdeg},
the boundary of each $A\in \SS$ has length $|C(A)|-\rdeg^*(A)$,
so $|V(A)|\le |C(A)|-\rdeg^*(A)$.
Let $\SS'$ be the set of $A\in \SS$ for which
$|V(A)|-(i-1)\rdeg^*(A)>0$.
Then we have
$|V_i|\le \sum_{A\in \SS'} \{|C(A)| - i\cdot\rdeg^*(A) \}$.
We may assume that $\SS'\not=\emptyset$.

By applying Lemma~\ref{L:inf} to $H$ and to each $A\in \SS'$ we obtain
$\rdeg^*(H)-\sum_{A\in \SS'}\rdeg^*(A) =
6(1-|\SS'|) - \sum(6-\deg(v))$ where the sum is taken over all vertices
$v$ that are in $H$, but not in or interior to any $A\in \SS'$.  Since $G$ has
maximum degree $6$, the sum is nonnegative, so
$-\sum_{A\in \SS'}\rdeg^*(A)\le -\rdeg^*(H) +6(1-|\SS'|)$.
Then we have
$|V_i|\le \sum_{A\in \SS'} |C(A)| - i\cdot\rdeg^*(H) +6i(1-|\SS'|)$.

Since $H$ is outerplane, we can put the induced cycles of $H$ in a reasonable order,
such that each new cycle contains at most two vertices already seen; then
$\sum_{A\in \SS'} |C(A)|\le |V(H)|+2(|\SS'|-1)$.
Therefore it suffices to show that $2(|\SS'|-1)+6i(1-|\SS'|)\le 0$.
Since $|\SS'|\ge 1$ and $i\ge 1$, $(2-6i)(|\SS'|-1)\le 0$.
\qed
\end{proof}

To end this section, we describe the structure of a triangulated cylinder and how it arises.

\begin{lemma} \label{C:cylinder0}
Suppose that $G$ is a plane triangulation of maximum degree 6 and
$W$ is a closed walk in $G$ that bounds a connected outerplane subgraph $H$
with $V(H)=V(W)$ and $E(H)=E(W)$ such that $W$ has either
either (i) no turns or (ii) exactly one right turn
and exactly one left turn.  Then either $H$ is a cycle of length $|W|$,
$H$ is a path of length $|W|/2$, or $H$ is the union of a path of length~$p$
and a cycle $B$ that intersect at one endpoint of the path,
such that $|B|=|W|-2p$, $0< p < |W|/2$, and $\rdeg^*(B) \ge 1$.
\end{lemma}

\begin{proof}
By the choice of $H$, it has no edges on its interior;
hence, any $2$-connected block of $H$ is a cycle.
If $H$ is $2$-connected, then $H$ is bounded by a cycle, so $H=W$.
Thus, we may assume that $H$ is not $2$-connected.

Let $v$ be a cut-vertex in $H$.  Let $k$ be the number of
blocks of $H$ that contain $v$; then $k\ge 2$.  Also, $v$ appears
exactly $k$ times in $W$.  Since $W$ has at most one right and left turn each,
$v$ is incident to $2k$, $2k-1$, or $2k+1$ edges on
the exterior of $H$.  Also, $v$ is incident to at least one edge in
each block of $H$, so $\deg(v) \ge (2k-1)+k=3k-1$.  Since $\deg(v)\le
6$, $k\le 7/3$.  Hence, $k=2$.
Let $B$ and $B'$ be the blocks of $H$ that contain $v$.

If $B$ is $2$-connected, then $B$ is a cycle, so $v$ is incident to
exactly two edges in $B$.  Since $v$ has degree at most~$6$ and
$v$ is incident to at least one edge in $B'$, $v$ is incident to
at most~$3$ edges that are not in $H$.  By the previous paragraph, $v$ is
incident to at least~$3$ edges on the exterior of $H$, since $2k-1=3$.
Therefore, $v$ is incident to exactly one edge in $B'$ and $v$ is
incident to exactly $3$ edges on the exterior of $H$.
Then $B'$ is a single-edge block,
and $\rdeg_i(W)=1$ and $\rdeg_j(W)=2$ for the two times where $W$
passes through $v$.  Since there is at most one $i$ such that
$\rdeg_i(W)=1$, it follows that $v$ is the only cut-vertex of $H$
in a $2$-connected block of $H$, $B$ is the only
$2$-connected block of $H$, and $B$ is a leaf-block of $H$.  Every
other cut-vertex of $H$ is in two single-edge blocks.

It follows that either $H$ is a path of length $p$ or $H$ is the union of
a cycle $B$ and a path of length $p$ that intersects $B$ at one of its endpoints.
Each edge in the path appears in $W$ twice, and each edge in $B$
appears in $W$ once, so if $H$ is a path then $2p=|W|$
and otherwise
$2p+|B|=|W|$.

If $H$ is the union of a path and a cycle $B$ that
intersect at $v$, then $v$ is incident to exactly $4$ edges on the
exterior of $B$:  $3$ edges on the exterior of $H$ and one edge in
the single-edge block of $H$ that contains $v$.
Then $B$ has a sharp left turn at $v$.  Other turns in $B$ are also
in $W$, so $\rdeg^*(B) \ge 1$.
\qed
\end{proof}

\begin{lemma} \label{C:cylinder1}
Suppose that $G$ is a plane triangulation of maximum degree 6 and
$C$ is a cycle with either (i) no turns or (ii) exactly one right turn
and exactly one left turn, and all its vertices have degree 6.
Let $S$ be the set of neighbors of $C$ that lie in the interior of $C$.

Then $C$ has no chords on its interior, $S\not=\emptyset$, and
$G[S]$ is a connected outerplane graph.
Also, $G[S]$ is bounded by a walk $W$ such that $|W|=|C|$
and the cyclic sequence of turns on $W$
has the same pattern of turns as the cyclic sequence of turns on $C$. 

Moreover, if every vertex in $S$ has degree~$6$, then $W$ is a cycle.
\end{lemma}

\begin{proof}
Let $H$ be the union of $C$ with all the chords of $C$ that lie in the interior of $C$.
Then $H$ is an outerplane subgraph of $G$.
If $H\not= C$, then the weak dual of $H$ is a nontrivial tree, with at least two
leaf-faces.  Let $uv$ be a chord incident to a leaf-face $f$ of $H$;
then $f$ is bounded by a cycle $B$, and $B-uv$ is a $u,v$-path
in $C$.
Since a face has at least three sides, this path contains
another vertex, $x$.
Since $\deg_G(x)=6$ and $\rdeg_i(C)\le 3$ for $x=v_i$, $x$ must be incident to an edge $xy$ in $f$,
and since $f$ contains no chords of $C$, $y$ must be in $f$.
Since $G$ is a triangulation, $u$ and $v$ must also be
incident to edges in $f$.  Let $f'$ be what remains of the interior of $C$ once $f$ and the edge $uv$ are removed.
Then $f'$ contains another leaf-face of $H$, and by the same reasoning as before
there is a vertex $z$ in $V(C)-V(B)$ with a neighbor in $f'$,
and either (i) $u$ and $v$ both have neighbors in $f'$, or
(ii) there is a chord incident to $u$ or $v$ in $f'$.
In case (i), $u$ and $v$ are each incident to three edges on the interior of $H$,
so they must both be right turns in $C$, a contradiction.
So we may assume we have case (ii) and that $v$ is incident to a chord of $H$.  Since $v$ is incident to
two chords of $H$ and an edge in $f$, $v$ must be a right turn of $C$, incident
to no other edges in the interior of $H$.  However, the argument could be applied
to a different leaf-face $f^*$ of $H$, and since there is only one right turn of $C$,
$v$ would have to be incident to an edge in $f^*$ as well.
This is a contradiction, so $H=C$.
That is, $C$ has no chords on its interior.

$C$ must be incident to edges on its interior, because $\deg_G(v_i)=6$ and
$\rdeg_i(C)\le 3$ for every vertex $v_i$ on $C$.
Therefore, $S\not=\emptyset$.
By Lemma~\ref{C:rdeg}, $G[S]$ is a connected outerplane graph bounded by
a walk $W$ with $|W|=|C|$. Note that when $C$ has no right or
left turns, the cyclic sequence of
marginal degrees of $W$ is the same as the cyclic sequence of the marginal degrees of $C$.
Then note that if we change $C$ to add a right turn and a left turn, $W$ gains a
right and a left turn in corresponding spots.

Suppose that $W$ is not a cycle.  By Lemma~\ref{C:cylinder0},
$G[S]$ has a path with an endpoint $v$ which has degree~1 in $G[S]$.
According to the sequence of turns along $W$, $v$ is incident
to 1, 2, or 3 edges on its exterior, so $v$ is incident
to at most~4 edges overall.  Thus, if every vertex of $S$
has degree~$6$ in $G$, then $W$ is a cycle.
\qed
\end{proof}

\begin{lemma} \label{C:cylinder2}
Suppose that $G$ is a plane triangulation of maximum degree~$6$ and
$C$ is an induced cycle with either (i) no turns or (ii) exactly one right turn
and exactly one left turn, and all its vertices have degree~$6$.
Let $j,j'$ be maximum such that
the vertices in the interior of $C$
with distance less than $j$ to $C$,
and the vertices exterior to $C$
with distance less than $j'$ to $C$,
all have degree~$6$.
Then $j,j'\ge1$, and $G$ contains the following
{\em triangulated cylinder}:

Let $w=|C|$.  Start with
the Cartesian product of a $w$-cycle and a path of length $\ell$,
with $j+j'-2\le \ell\le j+j'$.
The vertices can be labeled $z_{a,b}$ with $a$ in the cyclic
group $Z_w$ and $0 \le b\le \ell$.
For some fixed $0\le k<w$, and for each $0\le b< \ell$,
add an edge from $z_{a,b}$ to $z_{a+1,b+1}$ if $0\le a< k$, and
add an edge from $z_{a,b}$ to $z_{a-1,b+1}$ if $k< a\le w$.
All triangles (except those of the form $z_{0,b_0}, z_{1,b_1}, z_{2,b_2}$, when
$w=3$) are 3-faces of $G$.

The triangulated cylinder has $w(\ell+1)$ vertices.
Moreover, $n- w(\ell+1)\le w(w-1)$.
\end{lemma}

\begin{proof}
Starting with $i=0$ and $C_0=C$, apply Lemma~\ref{C:cylinder1} to $C_i$
to produce a closed walk $W_{i+1}$ on the neighbors of $C_i$ that lie in the interior of $C_i$.  If
every vertex of $W_{i+1}$ has degree~$6$, then it is a cycle;
rename it $C_{i+1}$ and repeat for $i:=i+1$.
This yields a sequence of $j-1$ cycles $C_0,\ldots,C_{j-1}$
and one walk $W_j$, where $W_j$ contains a vertex of degree less than~$6$.
Each one has length $w=|C|$ and the same turn pattern as $C$.
Note that these cycles (including $W_j$ if it is a cycle) and the edges between
them form a triangulated cylinder.

By Lemma~\ref{C:cylinder0}, $W_j$ is a cycle, $W_j$ bounds a path of
length $w/2$, or $W_j$ bounds
the union of a path of length $p$ and a cycle $B$
such that $|B|=w-2p$, $0< p < w/2$, and $\rdeg^*(B) \ge 1$.

Suppose that $W_j$ is a cycle.  Let $H$ be the outerplane graph induced by
the neighbors of $W_j$ on its interior.  By Lemma~\ref{C:rdeg}, if $H$
is nonempty then $H$ is connected,
$\rdeg^*(H)\ge \rdeg^*(W_j)+1=1$, and $|V(H)|\le w-1$.
Then by Lemma~\ref{L:thank goodness}, there are at most
$\sum_{i=0}^{w-1}(w-1-i) = w(w-1)/2$
vertices in $H$ or its interior; i.e., on the interior of $W_j$. Obviously if $H$ is empty this bound also holds on
the number of vertices on the interior of $W_j$.

If $W_j$ is a path, then it has $w/2+1$ vertices; note that
$w/2+1\le w(w-1)/2$.

Suppose that $W_j$ bounds the union of a path of length $p$ and a
cycle $B$ of length $b$
such that $b=w-2p$, $0< p < w/2$, and $\rdeg^*(B) = 1$.
By Lemma~\ref{L:thank goodness}, there are at most
$\sum_{i=0}^{b}(b-i) = b(b+1)/2$
vertices in $B$ or its interior.  Thus, the number of vertices
in $W_j$ and in its interior is at most $p+b(b+1)/2=w/2+b^2/2$.
Since $p\ge 1$, $b\le w-2$ and $w/2+b^2/2\le w/2+(w-2)^2/2
=(w^2-3w+4)/2\le w(w-1)/2$.

$G$ can be re{\"e}mbedded in the plane so that the interior and exterior
of $C$ are switched.  Then the previous discussion applies with $j$
replaced by $j'$.

The triangulated cylinder described in the statement of the lemma
is simply a description of what is produced, with $\ell=j+j'$ if
$W_j$ and $W_{j'}$ (on the interior and exterior of $C$, respectively)
are both cycles, $\ell=j+j'-1$ if one is a cycle, and $\ell=j+j'-2$ if
neither is a cycle.  There are at most $w(w-1)/2$ vertices on
the interior and exterior of the cylinder, 
so there are at most $w(w-1)$
vertices of $G$ that are not on the cylinder.
The number of vertices in the triangulated cylinder is $w(\ell+1)$.
\qed
\end{proof}

\section{Proof of Theorem~\ref{T:main}}\label{S:proof}

Let $G$ be an $n$-vertex plane triangulation of maximum degree 6,
with $n\ge n_0$, with $n_0$ a constant to be specified later.

Let $U$ be the set of vertices of degree less than 6.  Using
Euler's Formula one can check that $\sum_{u\in U}(\deg(u)-6)=-12$,
which implies that $|U|\le 12$, and (assuming that $n_0>3$)
also $|U|\ge 4$.

Let $T$ be a {\em Steiner tree} for $U$ in $G$; that is, let $T$ be a tree in $G$
such that $U\subseteq V(T)$ and $T$ is of minimum size.
If we let $L(T)$ be the set of leaves in $T$, then
$L(T)\subseteq U$.  One can prove by induction
that a tree with $k$ leaves has at most $2(k-1)$ vertices of degree
not equal to $2$.  Then if we let $U' = U\cup \{v\in V(T): \deg_T(v)\not=2\}$,
we have $|U'|\le 2(|L(T)|-1) + |U-L(T)| \leq 2|U|-2\le 22$.

Let $P$ be a longest path in $T$ such that no internal vertex is in $U'$.
(The endpoints of $P$ are in $U'$.)
The length of $P$, denoted $|P|$, is the number of edges in $P$.
Note that $n(T) = e(T)+1 \le 21\cdot|P|+1$.

Next, we define $G'$: make two copies of each edge of $T$ and, for each vertex $v\in V(T)$, make
$\deg_T(v)$ copies of $v$.  Draw these all near the original edges and vertices,
and create incidences in the natural way so that we obtain a plane
graph with one face $f_T$ that contains $T$ (before deleting $T$), and the other
faces are all $3$-faces (that correspond to the faces of $G$). See Figure~\ref{F:construct G'}.
Note the boundary of $f_T$ is a cycle; it cannot have repeated vertices since $T$ is acyclic.
For convenience, let us re{\"e}mbed
$G'$ in the plane such that $f_T$ is the outer face.
Note that if $f_T$ is deleted, we obtain a triangulated disc.
Let $V_T'$ be the vertices in $G'$ copied from $V(T)$; then $G'-V_T'=G-V(T)$.

\begin{figure}[ht]
\begin{center}
\includegraphics{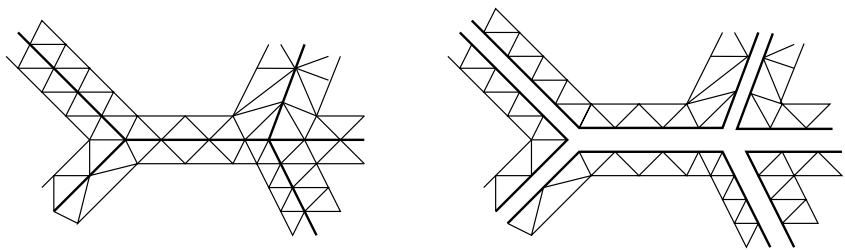}
\end{center}
\caption{An example of constructing $G'$ near a portion of $T$}
\label{F:construct G'}
\end{figure}

\begin{lemma}\label{C:map}
$G'$ can be mapped to $G_\infty$ such that vertices are sent to
vertices, edges to edges, and interior $3$-faces to $3$-faces,
such that adjacent $3$-faces in $G'$ are mapped to distinct $3$-faces in $G_\infty$.
\end{lemma}

\begin{proof}
A {\it facial triangle} of a plane graph is a triangle that bounds a $3$-face.
(Every triangle in $G_\infty$ is a facial triangle.)
Note that if two facial triangles of $G'$ share an
edge and one is already mapped to $G_\infty$, then $(*)$ there is exactly one way to
extend the map to the other triangle such that their union is mapped
isomorphically to the union of two triangles in $G_\infty$.

We begin by describing a map that sends
each facial triangle of $G'$ to a triangle in $G_\infty$,
isomorphically.  First, map one facial triangle $f_0$ of $G'$
to a triangle in $G_\infty$ arbitrarily. For each other facial triangle $f$ of $G'$,
consider a sequence of triangles $f_0,f_1,\ldots,f_i=f$ that forms a $f_0,f$-path
in the dual of $G'$ (i.e., consecutive triangles share an edge), and use
this and $(*)$ to determine how to map $f$ to $G_\infty$.  This
is well-defined if $f$ would be mapped identically to $G_\infty$ using any other
$f_0,f$-path in the dual of $G'$.  Two $f_0,f$-paths in the dual of $G'$
can be combined to get a sequence that begins and ends with $f$,
such that each pair of consecutive facial triangles satisfies $(*)$.  Thus, it
suffices to show that for any such sequence $W$ in $G'$ that begins and ends at the
same face $f$, the first and last copies of $f$ are mapped identically to $G_\infty$.

We will prove it by induction.
Note that a cycle in the plane dual of $G'$ is a Jordan curve,
so its interior may contain vertices of $G'$.
For any closed walk $W$ in the plane dual of $G'$,
let $h(W)$ be the total number of such vertices for all cycles in $W$.
We will do induction primarily on $h(W)$ and break ties according to the length of $W$.

Suppose that $f'$ is not first or last in $W$ and $f'$ is repeated in $W$.
Then we may let $W'$ be a $f',f'$-subsequence of $W$, and let $W''$ be the
$f,f$-sequence obtained by replacing $W'$ by $f'$ in $W$.
Since $h(W')+h(W'')\le h(W)$, we have $h(W')\le h(W)$ and $h(W'')\le h(W)$.
Also, the lengths of $W'$ and $W''$ are each strictly less than the length
of $W$.  Hence, we can apply induction to $W'$ and $W''$.
Thus, we may assume that $W$ has no repeated triangles
(other than its first and last triangle, $f$).

Write $W$ as $f=f_0',\ldots,f_\alpha'=f$.
If $\alpha=2$, then $f_0'=f_2'$, and applying $(*)$ twice at $f_1'$ shows that $f_0'$ and $f_2'$ are mapped
identically to $G_\infty$.
Otherwise the dual of $W$ is a cycle
in the plane dual of $G'$; as a Jordan curve it has a nonempty interior $R$.
Now, $f_0'$ and $f_1'$ share an edge; let $v$ be its endpoint that is in $R$.
Let $k$ be maximum such that $f_0',\ldots,f_k'$ are all incident to $v$; then $1\le k\le 5$.
By the construction of $G'$, $\deg_{G'}(v)=6$.
Thus we can replace the faces between $f_0'$ and $f_k'$ in $W$ by the other
$5-k$ triangles of $G'$ that are incident to $v$, and obtain a new closed walk $W'$.
We can apply induction to $W'$ because $h(W')= h(W)-1$.
This finishes the proof that for any such sequence $W$ in $G'$ that begins and ends at the
same face $f$, the first and last copies of $f$ are mapped identically to $G_\infty$.

So we have a well-defined map from $G'$ to $G_\infty$ that is isomorphic on each facial triangle,
and satisfies $(*)$.   For any two adjacent 3-faces $f,f'$ in $G'$, there is a walk in the dual of $G'$
from $f_0$ that ends with either $f,f'$ or $f',f$, and a shortest such walk will be a path.  Then
according to the definition of our map, $f$ and $f'$ must be mapped to distinct adjacent
triangles in $G_\infty$.  Since each edge $e\in E(G')$ is in one or two facial triangles, by $(*)$
the induced map on the edges is well-defined.
For each vertex $v\in V(G')$, the facial triangles that contain $v$ form
a path or a cycle in the dual of $G'$, and according to $(*)$ they will be mapped to triangles that
are consecutive around some vertex in $G_\infty$; it follows that there is a well-defined induced map on the vertices.
\qed
\end{proof}

Let $g$ be the above map.

\begin{lemma} \label{new 6}
If $g(x)=g(y)$ in $G_\infty$ where $x,y\in V(G')$ are distinct vertices, then their closed neighborhoods, $N[x]$ and $N[y]$, are
disjoint (or in other words, the distance from $x$ to $y$ in $G'$ is at
least~3).
\end{lemma}

\begin{proof}
If $x,y$ are adjacent in $G'$, then they are mapped to adjacent
(hence distinct) vertices by Lemma~\ref{C:map}.  Next suppose that $z\in N(x)\cap
N(y)$.  The $3$-faces of $G'$ that are incident to $z$ will either form
a path or a cycle in the dual of $G'$.  According to the map $g$, the
images of the faces under $g$ will again be consecutive around $g(z)$.
Since $\deg(z)\le 6$, this will map all neighbors of $z$ to distinct
neighbors of $g(z)$ in $G_\infty$.
\qed
\end{proof}

There is a pattern of vertices from $G_{\infty}$
that uses every seventh vertex (see the right side of Figure~\ref{F:infinite-grid}).
Let $S\subseteq V(G_{\infty})$ be the (infinite) set of vertices indicated in the figure.
For each vertex $v\in V(G')$, $g(v)\in S$ or $g(v)$ is adjacent to a vertex in $S$.  If $v\not\in V_T'$ then
the seven vertices of $N[v]$ map to the seven vertices of $N[g(v)]$; as this
includes one vertex of $S$, $v$ is dominated by $g^{-1}(S)=\{v\in V(G')\st g(v)\in S\}$.  Therefore $G'$ is dominated by the union
of $g^{-1}(S)$ and $V_T'$, or equivalently
the union of $g^{-1}(S)-V_T'$ and $V_T'$.
Let $D$ be the union of $g^{-1}(S)-V_T'$ and $V(T)$.  By the construction
of $G'$ from $G$, $D$ is a subset of $V(G)$ that dominates every vertex of $G$.

Next we consider the size of $D$.  For distinct $x,y\in g^{-1}(S)$,
the closed neighborhoods $N[x],N[y]$ in $G'$ are disjoint:
 if $g(x)=g(y)$, this holds by Lemma~\ref{new 6}, and if $g(x)\not=g(y)$, then
$N[g(x)]\cap N[g(y)]=\emptyset$ by definition of $S$,
and $N[x]\cap N[y]=\emptyset$ by Lemma~\ref{C:map} and the definition of $g$.
Then since $|N[x]|=7$ for each $x\in g^{-1}(S)-V_T'$,
the size of the set $g^{-1}(S)-V_T'$
is at most $|V(G')|/7$.  Since $|V_T'| =\sum_{v\in V(T)}\deg_T(v) = 2e(T) = 2n(T)-2$,
we obtain $|D|\le (n-n(T)+2n(T)-2)/7 + n(T) = n/7 + 8n(T)/7 - 2/7$.

Since $D$ dominates $G$, it suffices to show that $|D|\le n/4$ for $n>n_0$.
If $n/7 + 8n(T)/7 - 2/7\le n/4$, or $n(T)\le 3n/32 +1/4$, then this is satisfied.  So, we henceforth assume that $n(T)>3n/32 +1/4$,
which implies that $21\cdot|P|+1 >3n/32 +1/4$ and $n< 224\cdot|P|+8$.
Note that by choosing $n_0$ large enough, we can ensure that $|P|\ge \ell_0$, where $\ell_0$ is a constant to be determined later.

Recall that $P$ is a longest path in $T$ such that no internal vertices are in $U'$.
Let $x$ be a middle vertex of $P$, that is, let $x$ be a vertex on $P$ of distance $\lfloor{|P|/2}\rfloor$ from an
endpoint of $P$.
Let $N_i(x)$ be the set of vertices of $G$ with distance exactly $i$ from $x$,
let $N_i[x]$ be the set of vertices of $G$ with distance at most $i$ from $x$, and
let $G_i$ be the graph induced by $N_i[x]$.

Suppose that $u\in U\cap N_i(x)$ for some $i$.
Let $Q$ be an $x,u$-path of
length $i$ (in $G_i$\,).  Let $v_1,v_2$ be the endpoints of
the path $P$.
Without loss of generality, assume
that the $x,u$-path in $T$ contains $v_1$ (and not $v_2$).
By the choice of $P$, deleting the interior of its $x,v_1$-subpath from
$T$ gives a 2-component graph that contains $U$.  We could then add $Q$
to that graph to obtain a connected graph that contains $U$, and let
$T'$ be a spanning tree of it.  Now,
$n(T')\le n(T)-(\lfloor{|P|/2}\rfloor-1)+(i-1)$ and
$T$ is a Steiner tree, so $i\ge \lfloor{|P|/2}\rfloor$.
Thus, for all $i<\lfloor{|P|/2}\rfloor$,
every vertex in $N_i[x]$ has degree 6 in $G$.

Let $r$ be minimum such that $G_r$ is not a triangulated hexagon.
Then $G_{r-1}$ accounts for $1+\sum_{i=1}^{r-1} 6i$ distinct vertices, so $n> 1+6r(r-1)/2>3(r-1)^2$.
Then $r<1+\sqrt{n/3}$, so $3r+1<4+\sqrt{3n}<4+\sqrt{3(224\cdot|P|+8)}$,
which is at most $(|P|-1)/2$ if and only if $|P|^2-2706|P| - 15 \ge 0$.
Our choice of $n_0$ will imply that $|P|\ge 2707$,
so $3r+1< (|P|-1)/2 \le \lfloor{|P|/2}\rfloor$.  Therefore
every vertex in $N_{3r+1}[x]$ has degree 6 in $G$.

\begin{figure}[ht]
\begin{center}
\psfrag{x}{$x$}
\psfrag{w0}{$W_0$}
\psfrag{w1}{$W_1$}
\psfrag{w2}{$W_2$}
\psfrag{w3}{$W_3$}
\psfrag{w4}{$W_4$}
\psfrag{w5}{$W_5$}
\includegraphics{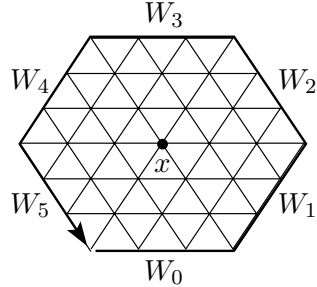}
\end{center}
\caption{$G_2$, incident edges, and $W$, when $r=3$. ($W$ is not necessarily a cycle.)}
\label{F:hex-grid}
\end{figure}

The union $N_{r-1}(x)\cup N_r(x)$ induces a cyclic sequence of triangles, such that
each pair of consecutive vertices along the boundary of $G_{r-1}$ has a common
neighbor in $N_r(x)$, and each of the six ``corner'' vertices of $G_{r-1}$ is
also adjacent to an additional (a third) vertex in $N_r(x)$ (called
{\em an $r$-corner}).
Also, this accounts for all vertices of $N_r(x)$,
and it induces a cyclic ordering of vertices in $N_r(x)$ that forms a
closed walk $W$ of length $6r$, which is formed from the conjunction of
walks $W_0,\ldots,W_5$ that begin and end at $r$-corners, such that
$G_{r-1}$ is to the left of each walk, and each
internal vertex of each walk has marginal degree 0.
(See Figure~\ref{F:hex-grid}.)
By the choice of $r$, $G_r$ is not a triangulated hexagon, so $N_r(x)$ does not induce a
cycle.  Therefore, either $(i)$ $W$ forms a cycle but not an induced cycle,
which is the case if nonconsecutive vertices along $W$ are adjacent, or
$(ii)$ there is at least one repeated vertex on $W$.

Because $G_r$ is part of a triangulation and all its vertices have
degree 6, the cases $(i)$ or $(ii)$ typically
will not occur in isolation.
For example, suppose that
$W_0=x_0,\ldots,x_r$ and $W_q=x_0',\ldots,x_r'$ share a vertex
$x_i=x_j'$ (this is case $(ii)$) with $0<i,j<r$.  Then we also must have
$x_{i-1}=x_{j+1}'$ and $x_{i+1}=x_{j-1}'$.  If, instead we have that $x_ix_j'$ is an
edge (case $(i)$) with $0<i,j<r$, then we also must have four more edges:
$\{x_{i-1}x_j', x_{i-1}x_{j+1}', x_ix_{j-1}', x_{i+1}x_{j-1}'\}$ or
$\{x_ix_{j+1}', x_{i-1}x_{j+1}', x_{i+1}x_j', x_{i+1}x_{j-1}'\}$.
Similar analysis (including special cases when $i$ or $j$ is $0$ or $r$)
and ignoring symmetric cases gives that
without loss of generality there are numbers $q,k$ with $0\le q\le 3$
and $0\le k\le r$ such that for $W_0$ and $W_q$, either $(i)$ $x_ax_b'$ is an edge for each
$0\le a,b\le k$ with $a+b\in\{k,k-1\}$,
or $(ii)$ $x_a=x_b'$ for each
$0\le a,b\le k$ with $a+b=k$.
See Figure~\ref{F:overlap}.

\begin{figure}[ht]
\begin{center}
\psfrag{W0}{$W_0$}
\psfrag{Wq}{$W_q$}
\psfrag{x0}{$x_0$}
\psfrag{xk}{$x_k$}
\psfrag{x01}{$x_0'$}
\psfrag{xk1}{$x_k'$}
\includegraphics{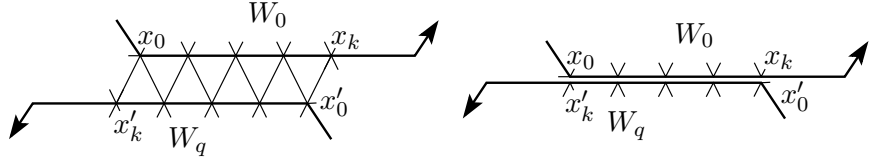}
\end{center}
\caption{$N_r(x)$ does not induce a cycle, cases $(i)$ and $(ii)$}
\label{F:overlap}
\end{figure}

Suppose that $q=0$. In case $(i)$ the edge $x_{\lfloor{k/2}\rfloor}x_{\lfloor{k/2}\rfloor}'$
will be a loop.
In case $(ii)$,
$x_{\lfloor{k/2}\rfloor}x_{\lceil{k/2}\rceil}'$ is a loop if $k$ is odd.
If $k$ is even and $k>0$ then we get $x_{k/2-1}=x_{k/2+1}'$ which
forces $x_{k/2}$ to have degree 3 in $G$.  If $k=0$ then $x_0$ and $x_0'$
are the same vertex with respect to $W$, so $x_0'$ is not actually a
repeated vertex on $W$ that satisfies $(ii)$.
Thus we may assume that $q>0$.

Suppose that $q=1$.  In case $(i)$ with $k=r$ and in case $(ii)$ with
$k=r-1$, $x_kx_0'$ is a loop.
Next consider case $(ii)$ with $k=r$.
Going clockwise around $x_r$ from $x_rx_{r-1}$, there is exactly one edge before $x_0'x_1'$ is reached.
However, $x_{r-1}=x_1'$ and $x_r=x_0'$, so
$x_rx_{r-1}=x_0'x_1'$ and $\deg_G(x_r)=2$, a contradiction.
Now consider the walk $W'=x_k,\ldots,x_r=x_0'$.
In case $(i)$ with $k<r$, we
add the edge $x_0'x_k$ from $G$ to make it a closed walk, then take it
in reverse order;
the walk obtained has turns at $x_k$ and $x_0'$, which are left turns.
In case $(ii)$ with $k\le r-2$,
$W'$ is a closed walk (since $x_0'=x_k$) with only one turn, a
sharp right at $x_k$; we reverse the order,
so that we obtain a walk with one (sharp) left turn.

If $q=2$ we take the straight diagonal walk from $x_0$ into $G_r$
and then make a right turn at the row containing $x_k'$, and go along
a straight horizontal walk to $x_k'$.  In case $(i)$ we add $x_k'x_0$, an edge of $G$,
to make it a closed walk, then reverse the order; in case $(ii)$ it
is a closed walk and we reverse the order.  We get a walk with exactly
one turn, which is a left turn.

Whether $q=1$ or $q=2$, the chosen walk is a cycle; call it $C_0$.
We obtain a contradiction
as follows:  $C_0$ has positive marginal degree.  As in Section~\ref{S:defs},
let $V_i$ be the set of vertices in the interior of $C_0$ at distance $i$
from $C_0$.  By Lemma~\ref{L:thank goodness}, we see that $V_i=\emptyset$
when $i\ge |C_0|$.  Since the interior of $C_0$ is triangulated, it or $C_0$ must contain at least one
vertex of $U$; otherwise, by Lemma~\ref{L:inf} the marginal degree of $C_0$ would equal 6,
necessitating more than two left turns.
 Hence there is a vertex of $U$ with distance less than $|C_0|$ from $C_0$.
Now if $q=1$, then $|C_0|\le r-k+1$ and the distance from $x$ to $C_0$ is $r$. If $q=2$, then $|C_0|\le 2r+1$ and $x$ is on $C_0$.
In either case, there is a vertex of $U$ at distance at most $3r+1$ from $x$, which contradicts the fact that every
vertex in $N_{3r+1}[x]$ has degree 6.

Thus we may assume that $q=3$.  Now it's easy to find a $x_0,x_k'$-path
through $x$ in $G_r$ with at most one turn (a left turn), and after
adding the edge $x_k'x_0$ if it's case $(i)$, we obtain a closed walk
with one right and one left turn (neither sharp) unless $k=0$ in
which case there are no turns.  This is a cycle of length $2r$ or $2r+1$, call it $C_0$.

Now we apply Lemma~\ref{C:cylinder2} with $C=C_0$.
Since every vertex of $C_0$ is distance at most $r$ from $x$,
no vertex of $U$, that is, no vertex of degree less than 6, is distance less than $\lfloor{|P|/2}\rfloor-r$
to $C_0$.
Therefore $j$ and $j'$ (as defined in Lemma~\ref{C:cylinder2})
are each at least $\lfloor{|P|/2}\rfloor-r$.
Thus we have a cylinder with $w=|C_0|$ and the other dimension
of length $\ell \ge 2 (\lfloor{|P|/2}\rfloor -r)-2 \ge |P|-2r-3$.
There are exactly $w(\ell+1)$ vertices in the cylinder and
$n\le w(\ell+1)+w(w-1) = w(\ell+w)$.

The vertices of the cylinder can be labeled $z_{a,b}$ with $a$ in the
cyclic group $Z_w$, and $0\le b\le \ell$, with edges as
described in Lemma~\ref{C:cylinder2}. Say {\it row $i$} to refer to the vertices
$z_{a,b}$ with $a=i$.

Let $H$ be a graph with vertices labeled $y_{a,b}$ with
$0\le a\le w+1$ and $0\le b\le \ell$, with $y_{a,b}$
adjacent to $y_{c,d}$ if (i) $|a-c|\le 1$ and (ii)
$z_{a\,{\rm mod\,}w,b}$ is adjacent to $z_{c\,{\rm mod\,}w,d}$.
Then $H$ can be embedded in $G_\infty$, and
$y_{a,b}\mapsto z_{a\,{\rm mod\,}w,b}$ maps $H$ to the cylinder
such that edges and faces are preserved.
Intuitively, this map wraps the ``sheet'' $H$ around the cylinder
such that rows $0$ and $1$ are glued to rows $w$ and $w+1$.

Let $S_H$ be the set of vertices in $H$ mapped to the usual dominating
set of $G_\infty$.  By the pattern of that set on $G_\infty$,
$S_H$ has at most $\lceil{\ell/7}\rceil$ vertices per row, so it
has size at most $\lceil{\ell/7}\rceil(w+2)$.  Note that for all
$a,b$ with $1\le a\le w$ and $1\le b\le \ell-1$, the vertex $y_{a,b}$
has six neighbors in $H$, just as in $G_\infty$, so it is dominated
by $S_H$.  Now let $S_Z$ be the set of vertices $z_{a,b}$
in the cylinder such that
there is some vertex $y_{c,b}\in S_H$ with $c\equiv a ({\rm mod\ }w)$.
Then $S_Z$ dominates all $z_{a,b}$ with
$1\le b\le \ell-1$, and $|S_Z|\le |S_H|$.
Then $S_Z\cup(V(G) - \{z_{a,b} : 1\le b\le \ell-1\})$ is a dominating set of $G$.
It has size at most $\lceil{\ell/7}\rceil(w+2) + n-w(\ell-1)$, so we will be satisfied if
this is at most $n/4$.

Equivalently, we would like to have
$3n/4\le w(\ell-1) - \lceil{\ell/7}\rceil(w+2)$.  Since
$n\le w(\ell+w)$ and $\lceil{\ell/7}\rceil\le (\ell+6)/7$, it follows
that $3w(\ell+w)/4 \le w(\ell-1) - (\ell+6)(w+2)/7$ would suffice.
This can be rewritten as $21w^2+52w+48\le (3w-8)\ell$,
then, since $3w-8$ is positive, as
$7w+36+336/(3w-8)\le \ell$.  Since $3w-8\ge 1$,
$7w+372\le \ell$ would suffice.
Now $w\le 2r+1$, so $14r+379\le \ell$ would suffice.
Then since $\ell\ge |P|-2r-3$,
$16r+382\le |P|$ would suffice.
Since $r<1+\sqrt{n/3}$ and $n<224\cdot|P|+8$,
it would suffice to have
$224(16\sqrt{n/3}+398)+8\le n$ be true.
This is equivalent to
$3584/\sqrt{3}+89160/\sqrt{n}\le \sqrt{n}$, so
we may complete the proof of Theorem~\ref{T:main}
by letting $n_0=4500000$.
\qed

\section{Extensions of Theorem~\ref{T:main} }\label{S:extend}

We conjecture that
using largely the same proof we gave for Theorem~\ref{T:main},
we could instead prove a sharp result, or prove the
same result for a more general class of graphs.

\begin{conjecture}\label{c:n/6}
There exists a constant $c$ such that any $n$-vertex triangulation with maximum degree $6$
has a dominating set of size at most $n/6 +c$.
\end{conjecture}

\begin{conjecture}\label{c:t}
For any constant $t$, there exists $n_t$ such that an $n$-vertex triangulation
with at most $t$ vertices of degree other than $6$ and $n\geq n_t$ has a dominating set of
size at most $n/4$.
\end{conjecture}

The major addition required to prove Conjecture~\ref{c:n/6} is a case analysis that,
for every triangulated cylinder from Lemma~\ref{C:cylinder2} with $3\le w\le 12$,
more carefully chooses $S_H$ and $S_Z$ above.
Recently, Liu and Pelsmajer claim to have verified both conjectures~\cite{LiPe}.

The following infinite family of graphs show that Conjecture~\ref{c:n/6} is best possible (up to the additive constant).

Let $n=6k$ for any positive integer $k$.  We construct a graph $G(n)$
on vertex set $\{v_1,\ldots,v_n\}$ as follows:
For  $1\le i\le n/3$ let the sets $S_i=\{v_{3i-2},v_{3i-1},v_{3i}\}$ induce
triangles, drawn concentrically in the plane.
For all $1\le j\le n-3$ add edges $v_jv_{j+3}$.
For all $1\le i<n/3$ add edges $v_{3i-2}v_{3i+2}$,
$v_{3i-1}v_{3i+3}$, $v_{3i}v_{3i+1}$.
Now $G(n)$ is a triangulation with maximum degree 6 (and exactly 6 vertices with
other degrees, each of degree 4).

\begin{proposition}
$G(n)$ has no dominating set with at most $n/6$ vertices.
\end{proposition}

\begin{proof}
For a contradiction, suppose that $G(n)$ has a dominating set $D$ of
size at most $n/6$.  It is not hard to see that for each $1\le i\le n/3$,
if $D$ does not intersect $S_i$ then $D$ must contain at least two
vertices of $S_{i-1}\cup S_{i+1}$.   For each $1\le i\le n/3$, let $c_i$
equal the number of vertices of $D$ in $S_i$ plus half the number of
vertices of $D$ in $S_{i-1}\cup S_{i+1}$.  By the previous observation,
$c_i\ge 1$ for all $1\le i\le n/3$.  Hence $\Sigma=\sum_{i=1}^{n/3}
c_i\ge n/3$.

Each vertex of $D$ contributes at most $2$ to $\Sigma$, so $\Sigma\le
n/3$.  Therefore $\Sigma$ must equal $n/3$, which implies that $c_i=1$ for all $i$ and every
vertex of $D$ contributes exactly $2$ to $\Sigma$.  However, then $D$
cannot intersect $S_1$, and $D$ will contain at most one vertex of
$S_2$.  Then $S_1$ is not dominated.
\qed
\end{proof}

\end{document}